\documentclass[12 pt,a4paper]{article}
\usepackage{graphicx} 
\usepackage{amsfonts}
\usepackage{amsmath}
\usepackage{amssymb}
\usepackage{amsthm}
\usepackage{tikz-cd}
\usepackage{mathrsfs}
\usepackage{float}
\usepackage{subfigure}
\usepackage{hyperref}
\usepackage{xcolor}
\usepackage{wrapfig}
\usepackage{geometry}
\usepackage{bm}
\usepackage{hyperref}
\hypersetup{colorlinks=true,
	linkcolor=red,
	anchorcolor=black,
	citecolor=blue,
	urlcolor=blue
}
\numberwithin{equation}{section}
\newtheorem{prop}{Proposition}
\newtheorem{thm}[prop]{Theorem}
\newtheorem{conj}[prop]{Conjecture}
\newtheorem{lemma}[prop]{Lemma}
\newtheorem{cor}[prop]{Corollary}

\theoremstyle{remark}
\newtheorem{rem}[prop]{Remark}

\DeclareMathOperator{\supp}{supp}

\def\R{{\mathbb R}}

\def\0{{\mathbf 0}}

\def\we{{\mathbf w}}

\def\vp{{\varphi}}
\newcommand{\qqs}{\mkern1mu}

\newcommand{\qqss}{\mkern2mu}

\def\www{{\widehat{\mathbf w}}}
\def\ooo{{\tilde{\omega}}}
\def\om{{\omega}}

\def\0{{\mathbf 0}}

\def\RRR{{\mathbb R}^2_{+}}
\def\BT{{B^1_{t_0}}}

\def\XXint#1#2#3{{\setbox0=\hbox{$#1{#2#3}{\int}$ }
\vcenter{\hbox{$#2#3$ }}\kern-.6\wd0}}

\begin{document}
\title{Sharp Basic Velocity Estimates for the Plane Steady Navier--Stokes Equations through Vorticity}
\author{Mikhail Korobkov\footnotemark[2]  \and Zhixiang Zhuang\footnotemark[3]}
\renewcommand{\thefootnote}{\fnsymbol{footnote}}
\footnotetext[2]{
School of Mathematical Sciences, Fudan University, Shanghai 200433, P.R.China. \newline email: korob@math.nsc.ru}
\footnotetext[3]{School of Mathematical Sciences, Fudan University, Shanghai 200433, P.R.China. \newline email: zxzhuang25@m.fudan.edu.cn}
\date{ }

\maketitle

\begin{abstract}We consider some new velocity estimates for general steady Navier-Stokes solutions in plane domains. 
Namely, we get some optimal estimates of   the difference between mean values of the velocity over two concentric circles in terms of the $L^2$-norm of the vorticity in the annulus between the circles. 
The sharpness of the getting estimates is demonstrated by some unexpected counterexamples.\end{abstract}

\section{Introduction}
In this paper, we study the stationary Navier--Stokes equations 
\begin{equation}  \tag{NS}\label{NS}
\left\{
\begin{aligned}
 & - \Delta \mathbf{w} + (\mathbf{w} \cdot \nabla) \mathbf{w} + \nabla p = \mathbf0 \ \ \mathrm{in} \ \Omega,\\
 & \nabla \cdot \mathbf{w} = 0 \ \ \mathrm{in} \ \Omega\\
\end{aligned}
\right.
\end{equation} 
in a plane domain (open connected set)  $\Omega\subset \R^2$,  which may be unbounded. Here $\we, p$ are the unknown velocity and pressure fields of the fluid.  
More specifically, we consider  $D$-solutions $\we$ to~(\ref{NS}) in $\Omega$, i.e., solutions with finite Dirichlet integrals 
\begin{equation} \label{unif-D-1}
\int_{\Omega} |\nabla \we|^2 < + \infty.
\end{equation}
Because of standard elliptic estimates, such solutions are $C^\infty$-smooth, and, moreover, real-analytic in $\Omega$. 

Still there are a lot of open problems here, for example, the famous {\it flow around an~obstacle problem}, solved for 3d case almost a century ago, but still open for 2d case for big Reynolds numbers (see, e.g.,  (\ref{NSobs}) for the formulation and \cite{KR23} for a~recent survey). Professor V.I.Yudovich ranked this problem second~--- immediately after the Millennium Problems~--- in his remarkable Cambridge lecture on the most important open problems in mathematical hydrodynamics (see~\cite{Yud03}\,). 

One reason why these exterior 2d problems are, perhaps paradoxically, much more difficult than their 3d counterparts is that the finiteness of the Dirichlet integral in the plane does not imply boundedness of the function, not even in an~integral-mean sense.
For instance, for  $f(z)=\bigl(\ln(2+|z|)\bigr)^{\frac13}$ we have $\int\limits_{\R^2}|\nabla f|^2<\infty$, but $f(z)\to+\infty$ as $|z|\to\infty$. This is in sharp contrast to the 3d case, where by Sobolev Embedding theorem one has
$$\biggl(\ \int\limits_{\R^3}|\nabla f|^2<\infty\biggr)\Rightarrow \biggl(\ \int\limits_{\R^3}|f-c|^6<\infty\biggr)$$ for some constant $c\in\R$. Really, the lack of the corresponding Embedding Theorem bears the main responsibility for the difficulties in the flow around an obstacle problem~(\ref{NSobs}) in 2d case and in the so called 'Stokes paradox' (see, e.g., \cite{R10}). Nevertheless, D.~Gilbarg and H.F.~Weinberger proved in~\cite{GW78} the uniform boundedness and the uniform convergence of the pressure for arbitrary $D$-solution. In particular, they established the following 
elegant estimate.

\begin{lemma}[\cite{GW78}] \label{lem-pmean}
{\sl
Let $\we$ be a $D$-solution to the Navier--Stokes equations \eqref{NS} in the annulus $ \Omega_{r_{1},r_{2}} =  \{z \in \R^2: r_1<|z|< r_2\}$. Then for the corresponding pressure $p$ we have
\begin{equation} \label{pmean}
|\bar{p}(r_2) - \bar{p}(r_1)| \le \frac{1}{2\pi} \int\limits_{\Omega_{r_1, r_2}} |\nabla \mathbf{w}|^2
\end{equation}
 where 
 $$\bar p(r)=\frac1{2\pi r}\int\limits_{|z|=r}p(z)\,ds$$
 is the mean value of the pressure~$p$ over the circle $S_r=\{z\in\R^2:|z|=r\}$.}
\end{lemma}

This is a really remarkable  result since under general assumption 
$f \in W^{1,2}(\Omega_{r_1,r_2})$   only the simple weaker estimate 
\begin{equation} \label{bardiff}
 |\bar{f}(r_2) - \bar{f}(r_1)| \le  \sqrt{\frac{1}{2\pi}\ln\frac{r_2}{r_1}}\left(\int\limits_{\Omega_{r_1, r_2}} |\nabla{f}|^2\right)^\frac12.
\end{equation}
holds with the logarithmic coefficient going to $+\infty$ as $\frac{r_2}{r_1}\to\infty$. 

The inequality (\ref{bardiff}) demonstrates the well-known fact: the $D$-function in general case may have a~logarithmic growth (for example, $f(z)=\bigl(\ln(2+|z|)\bigr)^{\alpha}$ with $\alpha\in (0,\frac12)$\,). 
Nevertheless, the brilliant structures of Navies--Stokes system allow to deduce much better estimates. The following result was obtained just recently. 

\begin{thm}[\cite{GKR23}, see also \cite{KR23,KR24}]\label{first basic estimate in annulus theorem}
{\sl Let $\mathbf{w}$ be the solution to the Navier-Stokes equations in the annulus $\Omega_{r_1,r_2}=\{r_1\le |z|\le r_2\}$. Then \begin{equation}\label{in:estim-m1}
|\overline{\mathbf{w}}(r_1)-\overline{\mathbf{w}}(r_2)|\le C\sqrt{\ln(2+\mu)\,D},
\end{equation} where 
$$\mu=\frac{1}{r_1m},\quad m=\max\limits_{r\in[r_1,r_2]}|\overline{\mathbf{w}}(r)|,\quad D=D_{r_1,r_2}=\int_{\Omega_{r_1,r_2}}|\nabla \mathbf{w}|^2,$$ and $C$ is some universal positive constant.}
\end{thm}

In our previous papers it was called ``the first basic estimate for the velocity". The estimate~(\ref{in:estim-m1}) looks similar to~(\ref{bardiff}) (both contain logarithmic factor), but really it has the different nature:
$\mu(r_1,r_2)$ does {\bf not} go to $+\infty$ as $r_2\to\infty$!  \ Thus, for $r_2 \gg r_1$, \eqref{in:estim-m1} significantly improves~(\ref{bardiff}) obtained  without using the Navier-Stokes equations. Moreover, the estimate~\eqref{in:estim-m1} is precise and can not be improved in general. Indeed, for a solution 
 to  the {\it flow around an obstacle problem}  
 \begin{equation}  \tag{OBS}\label{NSobs}
\left\{
\begin{aligned}
 & - \Delta \mathbf{w} + (\mathbf{w} \cdot \nabla) \mathbf{w} + \nabla p = \mathbf0 \ \ \mathrm{in} \ \Omega,\\
 & \nabla \cdot \mathbf{w} = 0 \ \ \mathrm{in} \ \Omega,\\
 & \mathbf{w}|_{\partial \Omega} = \mathbf0, \\
 & \mathbf{w} \to\mathbf{w}_\infty=\lambda \mathbf{e}_1 \ \ \text{as}\ \  |z| \to \infty.
\end{aligned}
\right.
\end{equation} 
in the complement to the unit ball $\Omega=\Omega_{1,\infty}=\R^2\setminus\overline{B}_1$ in the~case of small $\lambda$ the opposite inequality
 \begin{equation} \label{in:estim-small-l}\left(\int\limits_\Omega|\nabla\we|^2\le C\,\frac1{\ln(2+\frac1\lambda)}\lambda^2\right)\Leftrightarrow \biggl(\sqrt{\ln(2+\mu)\,D_{1,\infty}}\le C\,\bigl|\overline{\mathbf{w}}(1)-\overline{\mathbf{w}}(\infty)\bigr|\,\biggr)
  \end{equation}
 holds with $C=C(\Omega)$ (see~\cite{FS67,KR21}$\qqss$). In other words, the inequality~\eqref{in:estim-m1} has a~sharp form and the logarithmic factor there cannot be omitted in general. 
 
Nevertheless, (\ref{in:estim-m1}) can be improved under some additional assumptions, for example, if $\mathbf{w}$ is a solution in the whole disk $B_{r_2}$, then the logarithmic factor in (\ref{in:estim-m1}) can be dropped:

\begin{thm}[\cite{KR-b25}]\label{first basic estimate in balls theorem}{\sl
Let $\mathbf{w}$ be a D-solution to the Navier-Stokes equations in $B_{R}$, then for any $\Omega_{r_1,r_2}$ with $0\le r_1<r_2\le R$, there holds 
\begin{equation}\label{i-basic-estimate}
|\overline{\mathbf{w}}(r_2)-\overline{\mathbf{w}}(r_1)|\le C\,\sqrt{D_{r_1,r_2}} \,,
\end{equation}
where $C$ is a universal positive constant.}
\end{thm}

At first glance, the above two estimates look quite optimal and sharp. But this first impression 
    is deceptive for the following reason: the~estimates (\ref{in:estim-m1}), (\ref{i-basic-estimate}) control the mean value difference in terms of the $L^2$-norm of the {\it entire} velocity gradient. But in fact, for our solenoidal vector-field $\we=(w_1,w_2)$ only part of the gradient, namely, the vorticity
$$\omega(z)=\partial_yw_{1}-\partial_xw_{2},$$ is responsible for such a mean value difference! Indeed, by a classical formula, for any divergence-free D-function $\mathbf{w}$ in annulus $\Omega_{r_1,r_2}$, we have 
\begin{equation}\label{representation of the difference of the average by omega}
	\overline{\mathbf{w}}(\rho_2)-\overline{\mathbf{w}}(\rho_1)=\frac{1}{2\pi}\int_{\Omega_{\rho_1,\rho_2}}\frac{\omega(z)z^\bot}{|z|^2}
\end{equation}
(see, e.g., \cite[p.~388]{GW78} or \cite[p.206]{Sa99}; here for a point $z=(x,y)\in\R^2$ we denote $z^\bot=(x,y)^\bot=(-y,x)$).  This means that $|\overline{\mathbf{w}}(r_2)-\overline{\mathbf{w}}(r_1)|=0$ whenever $\omega=0$ in $\Omega_{r_1,r_2}$. However,  in (\ref{in:estim-m1}) and (\ref{i-basic-estimate}) right-hand side is not zero when $\omega=0$. Thus, the estimates (\ref{in:estim-m1}) and (\ref{i-basic-estimate})  are far from optimal when vorticity is relatively small.

The first two main results of the present paper state that the above estimates can be significantly strengthened to demonstrate the vorticity's leading role in controlling the mean value difference, in the case of balls and annulus respectively.

\begin{thm}\label{control by omega in balls}{\sl 
	Let $\mathbf{w}$ be a~solution to the Navier-Stokes equations in $B_R$, then for any  $\Omega_{r_1,r_2}=\{r_1\le |z|\le r_2\}$ with $0\le r_1<r_2\le R$, there holds 
\begin{equation}\label{eq:est-balls}|\overline{\mathbf{w}}(r_1)-\overline{\mathbf{w}}(r_2)|\le C\sqrt{\ln\left(2+\frac{D}{D_\omega}\right)D_\omega},
\end{equation}
	where $C$ is a universal positive constant and  $$D_\omega=\int_{\Omega_{r_1,r_2}}|\omega|^2,\quad\quad D=\int_{\Omega_{r_1,r_2}}|\nabla \mathbf{w}|^2.$$}
\end{thm}

\begin{thm}\label{control by omega in annulus}{\sl 
	Let $\mathbf{w}$ be a~solution to the Navier-Stokes equations in the annulus $\Omega_{r_1,r_2}=\{r_1\le |z|\le r_2\}$ with $0\le r_1<r_2\le R$, then  
\begin{equation}\label{eq:est-ann}|\overline{\mathbf{w}}(r_1)-\overline{\mathbf{w}}(r_2)|\le C\sqrt{\ln(2+\mu)\ln\left(2+\frac{D}{D_\omega}\right)D_\omega},
\end{equation}
where $C$ is a universal positive constant and $$\mu=\frac{1}{r_1m},\qquad m=\max\limits_{r\in[r_1,r_2]}|\overline{\mathbf{w}}(r)|,\qquad  D_\omega=\int_{\Omega_{r_1,r_2}}|\omega|^2,\quad\quad D=\int_{\Omega_{r_1,r_2}}|\nabla \mathbf{w}|^2.$$}
\end{thm}
\newenvironment{proofof3}{\begin{proof}[\indent\underline{\sc Proof of Theorem \ref{control by omega in balls}}. \ \ \ ]}{\end{proof}}
\newenvironment{proofof4}{\begin{proof}[\indent\underline{\sc Proof of Theorem \ref{control by omega in annulus}}. \ \ \ ]}{\end{proof}}

\begin{cor}\label{cor-power-est}{\sl 
	For any $\delta\in (0,1)$ the estimates (\ref{in:estim-m1}) and (\ref{i-basic-estimate}) stay true, if in their right-hand sides we replace $D$ by 
$D_\omega^\delta \, D^{1-\delta}$ (and in this case the universal constants in the right hand sides will depend on $\delta$ only as well). }
\end{cor}

Noting that there is an additional  factor $\left(2+ \ln(D/D_\omega)\right)^{\frac{1}{2}}$ on the right hand side of the estimates in Theorems~\ref{control by omega in balls}--\ref{control by omega in annulus}. Now we'd like to ask whether this factor can be removed, i.e., whether the difference between mean values of the velocity $\mathbf{w}$ over two concentric circles can be controlled only in terms of $\omega$ only. In other words, a natural question arises about the validity of the following

\begin{conj}\label{conjecture}{\sl Let $\mathbf{w}$ be the solution to the Navier-Stokes equations in $B_1$.  Then for any $0\le r_1<r_2\le 1$, $\Omega_{r_1,r_2}=\{r_1\le |z|\le r_2\}$, there holds 
$$|\overline{\mathbf{w}}(r_2)-\overline{\mathbf{w}}(r_1)|\le C\left(\int_{\Omega_{r_1,r_2}}|\omega|^2\right)^{\frac{1}{2}},$$
where $C$ is a universal positive constant.}
\end{conj}

Since vorticity plays a key role in many classical relations of two-dimensional hydrodynamics (see, for example, \cite{GW78}), and by virtue of many previous positive results, it is natural to expect the validity of the above statement. 
Surprisingly, Conjecture \ref{conjecture} fails:

\newenvironment{proofof6}{\begin{proof}[\indent\underline{\sc Proof of Theorem \ref{f-conj}}:]}{\end{proof}}
\begin{thm}\label{f-conj}{\sl 
There exists a sequence of $\{\mathbf{w}_k\}$ satisfying (\ref{NS}) in  $\Omega=B_1$, such that \begin{equation}\label{for-conj-f}
\frac{\bigl|\overline{\mathbf{w}}_k(1)-\overline{\mathbf{w}}_k(0)\bigr|}{\|\omega_k\|_{L^2(B_1)}}	\ =\ \frac1{2\pi}\left|\int_{B_1}\frac{\omega_kz^\bot}{|z|^2}\right|\Big/\|\omega_k\|_{L^2(B_1)}\to\infty,\quad \text{\rm as \ }k\to\infty.
\end{equation}}
\end{thm}

This means that Conjecture \ref{conjecture} fails with $r_2=1$ and $r_1=0$ !

Let us say something concerning the proofs and the main ideas behind. The~estimates in both 'positive' Theorems~\ref{control by omega in balls}--\ref{control by omega in annulus} are highly nontrivial, of course (compare, e.g., with the~standard estimate~(\ref{bardiff}), which is quite sharp for general $D$-functions). 
Nevertheless, the derivation of these two new results from Theorems~\ref{first basic estimate in annulus theorem}--\ref{first basic estimate in balls theorem} is very simple; only some accurate combination of elementary arguments from real and harmonic analysis is required, all the non‑triviality of these results is ''hidden'' in the previously proven Theorems~\ref{first basic estimate in annulus theorem}--\ref{first basic estimate in balls theorem}. 

Finding a counterexample is more tricky, since it isn't  easy in general to construct {\it an exact} solution to the unforced Navier--Stokes system  with prescribed paradoxical properties. 
First of all, as the estimate is given in terms of $\omega$ only, we will focus on the properties of vorticity only, and will look for a solution to the following profile system:
\begin{equation}\label{vor-ns-form}
	\begin{cases}
		-\Delta\omega+(\alpha\www_A+\mathbf{w}_B)\cdot\nabla \omega=0,\quad &\text{in }B_1,\\
		\omega=\omega_1,&\text{on }\partial B_1,\\
		\www_A=\nabla^{\bot}\psi_A,\quad &\text{in }B_1,\\
		\mathbf{w}_B(z)=\frac1{2\pi i}\int_{B_1} \frac{\omega(\zeta)}{z-\zeta}\,d\xi d\eta, \quad\zeta=\xi+i\eta, \quad &\text{in }B_1.
	\end{cases}
\end{equation}
where $\psi_A$ is a fixed harmonic function, 
$\omega_1$ is some smooth function defined in $\overline{B}_1$ and constant along the level set of $\psi_A$, $\alpha$ is an arbitrary positive constant (will be chosen later), and we naturally associate a two-dimensional vector~$\mathbf{w}_B=(u_B,v_B)$ with a complex number
$\mathbf{w}_B=u_B-iv_B$. Under this agreement, for $\we=\alpha\www_A+\we_B$ \ a\,beautiful relation 
\begin{equation}\label{bar-z-der}
\partial_{\bar z}\we=\partial_{\bar z}\we_B=\frac{i\,\omega}2
\end{equation}
holds (see, e.g., \cite{GW78}), and it is easy to check, that $\we$ satisfies the Navier--Stokes system~(\ref{NS}).  

We prove that under some extra assumptions on $\omega_1$, there exists a~weak solution to (\ref{vor-ns-form}) with an~unexpected apriory bound:

\begin{thm}\label{ex-omega1}  {\sl There exist some universal positive constant $C_*, C$ such that 
for any harmonic function~$\psi_A\in C^2(\bar B_1)$, a~real parameter $\alpha\in\R$, and for any Sobolev  function~$\omega_1\in H^1(B_1)$ meeting
\begin{equation}\label{apr-omeg1}
		\|\omega_1\|_{L^2(B_1)}\le C_*,\qquad \www_A\cdot\nabla\omega_1\equiv 0\mbox{\ \ \ in \,}B_1
	\end{equation} 
with $\www_A=\nabla^\bot\psi_A$, 
there exists at least one solution $\omega\in H^1(B_1)$ satisfying $(\ref{vor-ns-form})$ and the following energy estimate: 
\begin{equation}\label{vor-apr-est}
		\|\nabla\omega\|_{L^2(B_1)}\le C\left(\|\omega_1\|_{L^4(B_1)}\|\omega_1\|_{L^2(B_1)}+\|\nabla\omega_1\|_{L^2(B_1)}\right).
	\end{equation} }
\end{thm}

Note, that both constants $C$ and $C_*$ are independent of $\alpha$, $\psi_A$, $\omega_1$, etc. Moreover, the remarkable and surprising feature here is that the entire right hand side in the apriory estimate~(\ref{vor-apr-est}) is independent of~$\alpha$, in particular, we can take $\alpha\to+\infty$! The proof of Theorem~\ref{ex-omega1} is based on the Leray-Schauder fixed point theorem, following the classical approach on the existence of solutions to the stationary Navier-Stokes system in bounded domains (see, e.g., \cite[Chapter~3]{leraybook}).
\newenvironment{proofof5}{\begin{proof}[\indent\underline{\sc Proof of Theorem \ref{ex-omega1}}.\ \ \ ]}{\end{proof}}

Theorem~\ref{ex-omega1} allows us to construct a lot of solutions to the  Navier--Stokes system~(\ref{NS}) in the vorticity formulation~(\ref{vor-ns-form}). Having such a tool, it is easy to finish the proof of Theorem~\ref{f-conj}.
Take $\psi_A=xy$ and firstly find a suitable smooth~$\omega_1$ constant on level sets of~$\psi_A$, supported in the subdomain $B_1\cap\{xy> t_0\}$ \,(with $t_0>0$ small enough), and satisfying the~required inequality
\begin{equation}\label{c-ex-omeg1}\left|\int_{B_1}\frac{\omega_1 z^{\bot}}{|z|^2}\right|>>\|\omega_1\|_{L^2(B_1)}.
\end{equation}
Multiplying the~constructed $\omega_1$ by a~sufficiently small positive constant, we can easily get~(\ref{apr-omeg1}) without changing~(\ref{c-ex-omeg1}), so now Theorem~\ref{ex-omega1}
guarantees the~existence of solutions to (\ref{vor-ns-form}) with arbitrary~$\alpha>0$. Taking $\alpha=\alpha_k\to+\infty$, we get a sequence of solutions~$\omega=\omega_k$ to~(\ref{vor-ns-form}) with the uniform estimate~
\begin{equation}\label{c-ex-omeg1--}
\|\nabla\omega_k\|_{L^2(B_1)}\le C
\end{equation}
(see~(\ref{vor-apr-est})\,), where, as we emphasized before, the constant~$C$ is independent of~$\alpha_k\to+\infty$! Then it is easy to check by some elementary Real Analysis arguments, that $\|\omega_k-\omega_1\|_{L^2(B_1)}\to 0$, which by~(\ref{c-ex-omeg1}) justifies the~required estimate ~(\ref{for-conj-f})
for~$k$ large enough. This finishes the proof of Theorem \ref{f-conj}, which ensures us, that the basic estimate cannot be improved to be independent of the whole~$\nabla\mathbf{w}$. 

The paper is organized as follows. In Section~\ref{prel}, we give some preliminary results. In Section~\ref{sec:3-basic-est},  we prove the improved basic estimate in balls and annulus respectively. In Section~\ref{sec:4-exist}, we prove the existence Theorem~\ref{ex-omega1} 
for the vorticity equations (\ref{vor-ns-form}). In Section~\ref{sec:5-fin}, we prove Theorem \ref{f-conj}, which shows that Theorem \ref{control by omega in balls} is sharp.

\section{Preliminaries}\label{prel}

We use the standard notations for Sobolev space $H^1(\Omega)=W^{1,2}(\Omega)$ and $H^1_0(\Omega)=\overline{C^\infty_0(\Omega)}^{H^1}$. For $f,g\in H^1_0(\Omega)$, we denote the inner product by $(f,g)=\int_\Omega \nabla f\cdot \nabla g$.
Below for $0< r_1<r_2$, we denote the corresponding annulus domain by symbol  $$\Omega_{r_1,r_2}=\{z\in\Omega:r_1<|z|<r_2\}.$$
Further, for a function $f\in H^1(\Omega_{r_1,r_2})$, we denote by $D_f(r_1,r_2)$ the Dirichlet integral of $f$:
$$D_f(r_1,r_2)=\int_{\Omega_{r_1,r_2}}|\nabla f|^2.$$ We also denote by $\overline{f}(r)$ the mean value of $f$ over the circle $S_r$ $$\overline{f}(r)=\frac{1}{2\pi r}\int_{S_r}f\,ds.$$
Using this notation, recall some classical results for D-functions. 

\begin{lemma}\label{lem:d-est-gen}{\sl 
For any $f\in H^1(\Omega_{r_1,r_2})$, there holds 
$$|\overline{f}(r_2)-\overline{f}(r_1)|\le\frac{1}{\sqrt{2\pi}}\left(\ln\frac{r_2}{r_1}\right)^{\frac{1}{2}}\sqrt{D_f(r_1,r_2)}.$$}
\end{lemma}

For a solution $\we=(u,v)$ to (\ref{NS}), we denote the corresponding vorticity by the symbol $\omega=\partial_yu-\partial_x v$ . We consider a~suitable complexification for the vector $\mathbf{w}\in\R^2$ of the form $\mathbf{w}=u-iv$. Then by the classical 
Cauchy-Pompeiu formula, there holds
\begin{equation}\label{complexification of w}
\mathbf{w}(z)=\frac{1}{2\pi i}\left(\int_{\partial\Omega_{r_1,r_2}}\frac{\mathbf{w}(\zeta)}{\zeta-z}\,d\zeta+\int_{\Omega_{r_1,r_2}}\frac{\omega(\zeta)}{\zeta-z}\,d\xi d\eta\right),\quad z\in\Omega_{r_1,r_2}
\end{equation}
(see, e.g., \cite{GW78}).
We can rewrite the representation (\ref{complexification of w}) in the form $\we=\we_A+\we_B$, where 
\begin{equation}\label{wa and wb}
\mathbf{w}_A(z)=c_A+\frac{1}{2\pi i}\int_{\partial\Omega_{r_1,r_2}}\frac{\mathbf{w}(\zeta)}{\zeta-z}\,d\zeta,\qquad \mathbf{w}_B(z)=-c_A+\frac{1}{2\pi i}\int_{\Omega_{r_1,r_2}}\frac{\omega(\zeta)}{\zeta-z}\,d\xi d\eta,
\end{equation}
and the constant $c_A$ is chosen in such a way that $\we_A$ mean values are identically zero:
\begin{equation}\label{average of wa is 0}
\overline{\mathbf{w}}_A(r)=\0,\quad \forall r\in(r_1,r_2)
\end{equation}
(this is possible, since by construction $\mathbf{w}_A$ is a~complex-analytic function, and, consequently, satisfies  the~corresponding mean-value property). Recall the following classical isometry identity for $\mathbf{w}_B$:

\begin{lemma}[see, e.g., \cite{Astala-book}]\label{wb estimate} {\sl For $\mathbf{w}_B$ defined in (\ref{wa and wb}), there holds
$$\|\nabla \mathbf{w}_B\|_{L^2(\mathbb{R}^2)}=\|\omega\|_{L^2(\Omega_{r_1,r_2})}.$$}
\end{lemma}

Under above notation, we have, in particular, that $\mathbf{w}_A$ is harmonic as a real vector-field, and satisfies $$\nabla\cdot \mathbf{w}_A=0,\quad \nabla\times \mathbf{w}_A=0.$$ 
Wen need the  following elementary estimate on the maximum of $\mathbf{w}_A$:

\begin{lemma}\label{max w_A}
	{\sl Let $\mathbf{w}$ be the solution to the Navier-Stokes equations in the annulus $\Omega_{r_1,r_2}=\{r_1\le |z|\le r_2\}$, and $\mathbf{w}_A, \mathbf{w}_B$ defined as above in~(\ref{wa and wb}) satisfying (\ref{average of wa is 0}). Suppose in addition that $0<4 r_1< r_2$. Then there holds 
\begin{equation}\label{wae}
\sup_{\Omega_{2r_1,\frac{r_2}{2}}} |\mathbf{w}_A|\le 4\sqrt{D_A}, 
\end{equation}	
where $$D_A=\int_{\Omega_{r_1,r_2}}|\nabla \mathbf{w}_A|^2.$$
}
\end{lemma}
\begin{proof}
	Since $\int_{\Omega_{r_1,2r_1}}|\nabla \mathbf{w}_A|^2\le  D_A$, we have by mean value theorem 
	$$r_1\int_{S_{\tilde{r}_1}}|\nabla \mathbf{w}_A|^2\le D_A.$$ for some $\tilde{r}_1\in (r_1,2r_1)$. Then by Cauchy-Schwarz inequality, $$\int_{S_{\tilde{r}_1}}|\nabla \mathbf{w}_A|\le \sqrt{\frac{2\pi \tilde{r}_1}{r_1}}\,\sqrt{D_A}\le 4\sqrt{D_A}.$$
	Since $\mathbf{w}_A$ has zero mean value over each circle~$S_r$ (see~(\ref{average of wa is 0})\,), we get $$|\mathbf{w}_A|\le \int_{S_{\tilde{r}_1}}|\nabla \mathbf{w}_A|\le 4\sqrt{D_A}\quad\text{on \ }S_{\tilde{r}_1}.$$
	Similarly, $$|\mathbf{w}_A|\le 4\sqrt{D_A}\quad\text{on }S_{\tilde{r}_2}.$$
	for some $\tilde{r}_2\in(r_2/2,r_2)$. Then by maximum principle for harmonic function, $$|\mathbf{w}_A|\le 4\sqrt{D_A}\quad\text{in}\quad\overline{\Omega_{2r_1,\frac{r_2}{2}}}\subset\Omega_{\tilde{r}_1,\tilde{r}_2}.$$
\end{proof}

\medskip
This simple estimate and the Cauchy integral formula allow us to have pointwise control on the gradient~$\nabla\we_A(z)$, see~(\ref{cases-ann}), and it will be crucial for the subsequent proofs of Theorems~\ref{control by omega in balls}--\ref{control by omega in annulus}.
We need also the following simple observation from real analysis concerning the monotonicity of the right-hand side of estimates~(\ref{eq:est-balls})--(\ref{eq:est-ann}): 

\begin{lemma}\label{subd-mon}
	{\sl  For any quadruple of numbers $s,t,s',t'$ satisfying $0<s'<s$ and $0<t'<t $ we have 
\begin{equation*}\label{eq:mon1}
t'\cdot\ln\biggl(2+\frac{s'}{t'}\biggr)<t\cdot\ln\biggl(2+\frac{s}{t}\biggr)
\end{equation*}}
\end{lemma}
\begin{proof}
	Monotonicity with respect to~$s$ is evident. So we can reduce the consideration to the model case $s=s'=1$. Then the desired inequality can be checked just taking a~derivative with respect to~$t$.
\end{proof}

For $\mathbf{w}$, we do not distinguish between real vector field and complex function. We use $C$ to denote constants that are universal. The exact values of $C$ may change from line to line.

\section{Improvement of the basic estimate to the control by~$\omega$}
\label{sec:3-basic-est}

\begin{proofof3}
	We firstly make some simplifications.
	\begin{itemize}
		\item[(\romannumeral1)] $D>100\,D_\omega$. Otherwise, if $D\le 100\,D_\omega$, then by Theorem~\ref{first basic estimate in balls theorem} we have  that 
		$$|\overline{\mathbf{w}}(r_2)-\overline{\mathbf{w}}(r_1)|\le C\sqrt{D}\le  10 C\sqrt{D_\omega}\le C'\sqrt{\ln\left(2+\frac{D}{D_\omega}\right)D_\omega}.$$
		
		\item[(\romannumeral2)] $r_1<\frac{D_\omega}{D}r_2$. Otherwise, if $\frac{r_2}{r_1}\le\frac{D}{D_\omega}$, then by (\ref{representation of the difference of the average by omega}) and H\"older inequality, we have that 
		$$\begin{aligned}
			|\overline{\mathbf{w}}(r_2)-\overline{\mathbf{w}}(r_1)|=\frac{1}{2\pi}\left|\int_{\Omega_{r_1,r_2}}\frac{\omega z^\bot}{|z|^2}\right|&\le \frac{1}{\sqrt{2\pi}}\left(\int_{\Omega_{r_1,r_2}}|\omega|^2\right)^{\frac{1}{2}}\sqrt{\ln\frac{r_2}{r_1}}\\
			&\le \sqrt{\frac1{2\pi}\ln\frac{D}{D_\omega}}\left(\int_{\Omega_{r_1,r_2}}|\omega|^2\right)^{\frac{1}{2}}.
		\end{aligned}$$
		\item[(\romannumeral3)] It suffices to prove that \begin{equation}\label{reduction 3}
	\left|\overline{\mathbf{w}}\left(\frac{r_2}{4}\right)-\overline{\mathbf{w}}(4r_1)\right|\le C\sqrt{\ln\left(2+\frac{D}{D_\omega}\right)D_\omega}.
		\end{equation} 
		
Indeed, if (\ref{reduction 3}) is established, then by (\ref{representation of the difference of the average by omega}) and H\"older inequality, we have that 
		$$	|\overline{\mathbf{w}}(r_1)-\overline{\mathbf{w}}(4r_1)|=\frac{1}{2\pi}\left|\int_{\Omega_{r_1,4r_1}}\frac{\omega z^\bot}{|z|^2}\right|\le \frac12\left(\int_{\Omega_{r_1,4r_1}}|\omega|^2\right)^{\frac{1}{2}}\le \frac12\sqrt{D_\omega}.$$
		Similarly, we have that $|\overline{\mathbf{w}}(r_2)-\overline{\mathbf{w}}(r_2/4)|\le \frac12\sqrt{D_\omega}.$ Hence 
		$$\begin{aligned}
\left|\overline{\mathbf{w}}(r_1)-\overline{\mathbf{w}}(r_2)\right|&\le \left|\overline{\mathbf{w}}(r_1)-\overline{\mathbf{w}}(4r_1)\right|+\left|\overline{\mathbf{w}}(4r_1)-\overline{\mathbf{w}}\left(\frac{r_2}{4}\right)\right|+\left|\overline{\mathbf{w}}\left(\frac{r_2}{4}\right)-\overline{\mathbf{w}}(r_2)\right|\\
&\le \sqrt{D_\omega}+C\sqrt{\ln\left(2+\frac{D}{D_\omega}\right)D_\omega}\le C' \sqrt{\ln\left(2+\frac{D}{D_\omega}\right)D_\omega}.
\end{aligned}$$
\end{itemize}
	
	From (\romannumeral1) and (\romannumeral2), we know that $r_2>100r_1$. In particular, for the geometric mean $r_0:=\sqrt{r_1r_2}$ we have
$$\frac{r_0}{r_1}, \,\frac{r_2}{r_0}>10.$$
Also from (\romannumeral1), \ 
\begin{equation}\label{add-ineq}
\frac{4}{5}D_A\le D\le \frac{5}{4}D_A\qquad\text{ and }\qquad81\,D_\omega\le D_A.\end{equation}
Now, after the above simplifications, we are ready to start the proof of the theorem. First of all, we have to get pointwise estimates for the gradient of the analytic part $\nabla\we_A$. 
Noting that for $z\in\Omega(4r_1,\frac{r_2}{4})$,
	$$\begin{aligned}
		|\nabla\mathbf{w}_A(z)|&=\frac{1}{2\pi}\left|\int_{\partial B_{\frac{r_2}{2}}}\frac{\mathbf{w}_A(\zeta)}{(z-\zeta)^2}\,d\zeta-\int_{\partial B_{2r_1}}\frac{\mathbf{w}_A(\zeta)}{(z-\zeta)^2}\,d\zeta\right|\\
		&\overset{(\ref{wae})}\le \frac2\pi\sqrt{D_A}\left(\int_{\partial B_{\frac{r_2}{2}}}\frac{1}{|z-\zeta|^2}+\int_{\partial B_{2r_1}}\frac{1}{|z-\zeta|^2}\right)\\
		&\le C\sqrt{D_A}\left(\frac{1}{r_2}+\frac{r_1}{|z|^2}\right).
	\end{aligned}$$
Then we have that 
\begin{equation}\label{cases-ann}
|\nabla\mathbf{w}_A(z)|\le\begin{cases}
		\frac{C\sqrt{D_A}r_1}{|z|^2},\quad &z\in\Omega_{4r_1,r_0};\\[8pt]
		\frac{C\sqrt{D_A}}{r_2},& z\in\Omega_{r_0,\frac{r_2}{4}}.
	\end{cases}\end{equation}
Thus, we are going to estimate separately 
$\left|\overline{\mathbf{w}}(\frac{r_2}{4})-\overline{\mathbf{w}}(r_0)\right|$ and $\left|\overline{\mathbf{w}}(r_0)-\overline{\mathbf{w}}(4r_1)\right|$.
	
	 Take some $r_\ast\in(4r_1, r_0)$ to be determined later. Then by (\ref{i-basic-estimate})--(\ref{representation of the difference of the average by omega}) and H\"older inequality,
	$$\begin{aligned}
		\left|\overline{\mathbf{w}}(r_0)-\overline{\mathbf{w}}(4r_1)\right|&\le 
		\left|\overline{\mathbf{w}}(r_\ast)-\overline{\mathbf{w}}(4r_1)\right|+\left|\overline{\mathbf{w}}(r_0)-\overline{\mathbf{w}}(r_\ast)\right|\\
		&\le \sqrt{D_\omega}\sqrt{\ln\frac{r_\ast}{4r_1}}+C\sqrt{D(r_\ast, r_0)}\\
		&\le \sqrt{D_\omega}\sqrt{\ln\frac{r_\ast}{4r_1}}+C\sqrt{D_A(r_\ast, r_0)}+C\sqrt{D_\omega(r_\ast, r_0)}.
	\end{aligned}$$
Here and below we use the natural notation $D(r_\ast, r_0)=\int\limits_{\Omega_{r_\ast, r_0}}|\nabla\we|^2$, \ \ $D_A(r_\ast, r_0)=\int\limits_{\Omega_{r_\ast, r_0}}|\nabla\we_A|^2$, etc.
By the first case in (\ref{cases-ann}) we have  
$$D_A(r_\ast,r_0)\le C\,\frac{r_1^2}{r^2_\ast}D_A.$$
	Therefore, 
	$$|\overline{\mathbf{w}}(r_0)-\overline{\mathbf{w}}(4r_1)|\le C\sqrt{\ln\frac{r_\ast}{4r_1}}\sqrt{D_\omega}+C\,\frac{r_1}{r_\ast}\sqrt{D_A}\qquad\text{ whenever }\ 8r_1<r_*<r_0.$$ 
Choose now $r_\ast=\sqrt{\frac{D}{D_\omega}}r_1.$
Note, that by (\ref{add-ineq}), \ $\frac{r_1}{r_\ast}\sqrt{D_A}\le2\sqrt{D_\omega}$, and by starting assumptions~(i)--(ii) we have that $8r_1<r_*<r_0$, consequently, 
\begin{equation}\label{first-ann-concl}
|\overline{\mathbf{w}}(4r_1)-\overline{\mathbf{w}}(r_0)|\le C\sqrt{\ln\frac{D}{D_\omega}}\sqrt{D_\omega}\end{equation}
as required.

Now take some $r_{\ast\ast}\in\left(r_0,\frac{r_2}4\right)$ to be determined later. Following the similar calculations, by (\ref{i-basic-estimate})--(\ref{representation of the difference of the average by omega}) and H\"older inequality,
	$$\begin{aligned}
		\left|\overline{\mathbf{w}}\left(\frac{r_2}4\right)-\overline{\mathbf{w}}(r_0)\right|&\le \left|\overline{\mathbf{w}}\left(\frac{r_2}4\right)-\overline{\mathbf{w}}(r_{\ast\ast})\right|+\left|\overline{\mathbf{w}}(r_{\ast\ast})-\overline{\mathbf{w}}(r_0)\right|\\
		&\le C\sqrt{D_\omega}\sqrt{\ln\frac{r_2}{4r_{\ast\ast}}}+C\sqrt{D(r_{\ast\ast}, r_0)}\\
		&\le C\sqrt{D_\omega}\sqrt{\ln\frac{r_2}{4r_{\ast\ast}}}+C\sqrt{D_A(r_{\ast\ast}, r_0)}+C\sqrt{D_\omega(r_{\ast\ast}, r_0)}.
	\end{aligned}$$
Then by the second case in (\ref{cases-ann}) we have  
$$D_A(r_{\ast\ast},r_0)\le C\,\frac{r_{\ast\ast}^2}{r^2_2}D_A.$$
	Therefore, 
	$$\left|\overline{\mathbf{w}}\left(\frac{r_2}4\right)-\overline{\mathbf{w}}(r_0)\right|\le C\sqrt{\ln\frac{r_2}{4r_{\ast\ast}}}\sqrt{D_\omega}+C\,\frac{r_{\ast\ast}}{r_2}\sqrt{D_A}\qquad\text{ whenever }\ r_0<r_{**}<\frac{r_2}8.$$ 
Choose now $r_{\ast\ast}=\sqrt{\frac{D_\omega}{D}}r_2$.
Note, that by  starting assumptions~(i)--(ii) we have that $r_0<r_*<\frac{r_2}8$ \ and \ $\frac{r_{\ast\ast}}{r_2}\sqrt{D_A}\le 2\sqrt{D_\omega}$ \ (see~(\ref{add-ineq})\,), consequently, 
\begin{equation}\label{second-ann-concl}
\left|\overline{\mathbf{w}}(r_0)-\overline{\mathbf{w}}\left(\frac{r_2}4\right)\right|\le C\sqrt{\ln\frac{D}{D_\omega}}\sqrt{D_\omega}.\end{equation}

Finally, (\ref{first-ann-concl})--(\ref{second-ann-concl}) imply the~required estimate~(\ref{reduction 3}). The proof of Theorem \ref{control by omega in balls} is finished.
\end{proofof3}

\begin{proofof4} Here we have two logarithmic factors, so the situation is a bit more delicate. We need some additional simplifications in the beginning in order to stabilize the situation with the~parameter~$\mu=\frac1{m\qqs r_1}$, where, recall, 
 $m=\max\limits_{r\in[r_1,r_2]}|\overline{\mathbf{w}}(r)|$. 
We claim that it is enough to consider a special case when 
\begin{equation}\label{eq:mu-stable}
\min\limits_{r\in[r_1,r_2]}|\overline{\mathbf{w}}(r)|\ge\frac15 m.
	\end{equation}
Indeed, if the above inequality fails, we can take $r_0\in[r_1,r_2]$ with maximal value $|\overline{\mathbf{w}}(r_0)|=m$, and also take $\tilde r\in[r_1,r_2]$ satisfying
$$|\overline{\mathbf{w}}(\tilde r)|=\frac15m, \qquad |\overline{\mathbf{w}}(r)|\ge\frac15m\ \ \mbox{ for all $r$ between $\tilde r$ and $r_0$}.$$
Assume for definiteness that $\tilde r<r_0$. Then by construction the analog of the requirement  (\ref{eq:mu-stable}) is fulfilled in the sub-annulus $\Omega(\tilde r, r_0)$. The corresponding coefficient $\tilde\mu=\mu_{\tilde r, r_0}=\frac1{m\qqs \tilde r}$ will be {\it smaller}, than the initial coefficient~$\mu$. Moreover, 
$$|\overline{\mathbf{w}}(r_0)-\overline{\mathbf{w}}(\tilde r)|\ge\frac45 m\ge\frac25|\overline{\mathbf{w}}(r_1)-\overline{\mathbf{w}}(r_2)|.$$ 
Thus, if we prove the analog of~(\ref{eq:est-ann}) in the reduced annulus $\Omega(\tilde r, r_0)$, i.e., that 
\begin{equation}\label{eq:est-ann-mm}|\overline{\mathbf{w}}(r_0)-\overline{\mathbf{w}}(\tilde r)|\le \tilde C\sqrt{\ln(2+\tilde\mu)\ln\left(2+\frac{D(\tilde r,r_0)}{D_\omega(\tilde r,r_0)}\right)D_\omega(\tilde r,r_0)},
\end{equation}
it implies the validity of the required estimate~(\ref{eq:est-ann}) in the original domain $\Omega(r_1,r_2)$ with a~slightly modified constant~$C=\frac52\tilde C$ (here we use also the monotonicity property of Lemma~\ref{subd-mon}). So below we can assume w.l.o.g. that~(\ref{eq:mu-stable}) is fulfilled. 

But then for any pair $r_1<r'< r''<r_2$ we have immediately that 
$$\mu'\le 5\mu,$$
where \,$\mu'=\mu_{r',r''}=\frac1{r'm'}$ \ with \ $m'=\max\limits_{r\in[r',r'']}|\overline{\mathbf{w}}(r)|$. 

Now, having such a~stabilization of the additional coefficient~$\mu$ in an~arbitrary sub-annulus, we can proceed with the proof of Theorem~\ref{control by omega in annulus} in exactly the same way as in the previous Theorem~\ref{control by omega in balls}. The only difference is that we must now use Theorem~\ref{first basic estimate in annulus theorem} instead of Theorem~\ref{first basic estimate in balls theorem} in our arguments; everything else remains exactly the same, as the appearance of the additional factor involving $\mu$ has no effect, by virtue of its aforementioned stability. So the proof
of Theorem~\ref{control by omega in annulus}  is finished.

\end{proofof4}

\section{Existence theorem for the vorticity equation}
\label{sec:4-exist}

\begin{proofof5} Fix a harmonic function~$\psi_A\in C^2(\bar B_1)$,  a~Sobolev  function~$\om_1\in H^1(B_1)$, and a~real parameter $\alpha\in\R$  from the assumptions of Theorem~\ref{ex-omega1}. 
Denote respectively $\www_A=\nabla^\bot\psi_A$.  
In order to reduce the issue to the homogeneous case, we are looking for a~solution $\omega\in\ H^1(B_1)$ to the vorticity equations (\ref{vor-ns-form}) in the form~$\om=\ooo+\om_1$ with the~unknown function~$\ooo\in H^1_0(B_1)$ satisfying the~following nonlinear system:
\begin{equation}\label{example for the basic estimate for ns, refined}
		\begin{cases}
			-\Delta\tilde{\omega}+(\alpha\www_A+\mathbf{w}_B)\cdot\nabla \tilde{\omega}=\Delta\omega_1-\mathbf{w}_B\cdot\nabla\omega_1,\quad &\text{in }B_1,\\
			\tilde{\omega}=0,&\text{on }\partial B_1,\\
		\end{cases}
	\end{equation}
where $\we_B$ is calculated through the target function~$\omega$ by the corresponding integral operator from~(\ref{vor-ns-form}):
$$\mathbf{w}_B(z)=\frac1{2\pi i}\int_{B_1} \frac{\omega(\zeta)}{z-\zeta}\,d\xi d\eta, \qquad\zeta=\xi+i\eta.$$
Note, that in derivation of~(\ref{example for the basic estimate for ns, refined}) we essentially used  the identity~$\www_A\cdot\nabla\omega_1\equiv0$ from the~assumptions of Theorem~\ref{ex-omega1}. 

Below our strategy is quite standard: we reduce the system~(\ref{example for the basic estimate for ns, refined}) to the equation in the Hilbert space $H^1_0$ with some compact continuous operator, and then apply the Leray--Schauder fixed-point theorem (see, e.g., \cite[Chapter~3]{leraybook}; really, here 
the~calculations are much simpler since we deal with real-valued functions). The~desired a~priory estimate on~$\ooo$ will follow from the~special structure of~(\ref{example for the basic estimate for ns, refined}), see below. 

\medskip
{\sc Step 1}. We recall the standard definition of weak solution to the system~(\ref{example for the basic estimate for ns, refined}): \begin{equation}\label{notion of weak solution when constructing the example}
		\int_{B_1}\nabla\tilde{\omega}\cdot\nabla\eta-\int_{B_1}\tilde{\omega}(\alpha\www_A+\mathbf{w}_B)\cdot\nabla\eta-\int_{B_1}\omega_1 (\mathbf{w}_B\cdot\nabla\eta)+\int_{B_1}\nabla\omega_1\cdot\nabla\eta=0,\qquad \forall\eta\in H^1_0(B_1).
	\end{equation}      
	By H\"older inequality, we have that for any $\tilde{\omega}\in H^1_0(B_1)$,
	\begin{equation}\label{l4 norm for omega}
		\left|\int_{B_1}\frac{\tilde{\omega}(\zeta)}{z-\zeta}\,d\xi d\eta\right|\le\|\tilde{\omega}\|_{L^4(B_1)}\left\|\frac{1}{z-\cdot}\right\|_{L^{\frac{4}{3}}(B_2(z))}\le C\|\tilde{\omega}\|_{L^4(B_1)}
	\end{equation}
	Therefore, we have that
	\begin{equation}\label{linfty forn wb}
		|\mathbf{w}_B(z)|= \frac1{2\pi}\left|\int_{B_1}\frac{\tilde{\omega}(\zeta)+\omega_1(\zeta)}{z-\zeta}\,d\xi d\eta\right|\le C\|\tilde{\omega}\|_{L^4(B_1)}+C\|\omega_1\|_{L^4(B_1)},\quad \forall z\in B_1.
	\end{equation}
	Hence, by H\"older inequality, Poincar\'e inequality, and Sobolev embedding $H^1_0(B_1)\hookrightarrow L^4(B_1)$,
	$$\begin{aligned}
		\left|\int_{B_1}\tilde{\omega}(\alpha\www_A+\mathbf{w}_B)\cdot\nabla\eta\right|&\le \|\nabla\eta\|_{L^2(B_1)}\|\tilde{\omega}\|_{L^2(B_1)}(\alpha\|\mathbf{w}_A\|_{L^\infty(B_1)}+\|\mathbf{w}_B\|_{L^\infty(B_1)})\\
		&\le C\|\nabla\eta\|_{L^2(B_1)}\|\nabla\tilde{\omega}\|_{L^2(B_1)}(\alpha\|\www_A\|_{L^\infty(B_1)}+\|\omega_1\|_{L^4(B_1)}+\|\tilde{\omega}\|_{L^4(B_1)})\\
		&\le C\|\nabla\eta\|_{L^2(B_1)}\|\nabla\tilde{\omega}\|_{L^2(B_1)}\left(\alpha\|\www_A\|_{L^\infty(B_1)}+\|\omega_1\|_{L^4(B_1)}+\|\nabla\tilde{\omega}\|_{L^2(B_1)}\right).
	\end{aligned}$$
Thus,  for any fixed $\ooo\in H^1_0(B_1)$ the mapping $H^1_0(B_1)\ni\eta\mapsto\int_{B_1}\tilde{\omega}(\alpha\www_A+\mathbf{w}_B)\cdot\nabla\eta$ defines a bounded linear functional in $H^1_0(B_1)$. By the~Riesz representation theorem, there exists a~unique element $\vp_A\in H^1_0(B_1)$ (depending on~$\ooo$, of course) such that $$\int_{B_1}\tilde{\omega}(\alpha\mathbf{w}_A+\mathbf{w}_B)\cdot\nabla\eta=\int_{B_1}\nabla\vp_A\cdot\nabla \eta:=\langle\vp_A,\eta\rangle_{H_0^1(B_1)}\qquad\forall \eta\in H^1_0(B_1).$$
	Similarly, there exist unique elements $\vp_B, \vp_1\in H^1_0(B_1)$ such that
	$$\int_{B_1}\omega_1 (\mathbf{w}_B\cdot\nabla\eta)=\langle\vp_A,\eta\rangle_{H_0^1(B_1)}\qquad\forall \eta\in H^1_0(B_1),$$
	$$-\int_{B_1}\nabla\omega_1 \cdot\nabla\eta=\langle\vp_1,\eta\rangle_{H_0^1(B_1)}\qquad\forall \eta\in H^1_0(B_1).$$
	Thus,  if we define the nonlinear operator $T$ from the Hilbert space $H^1_0(B_1)$ to itself  by the formula $T(\ooo):=\vp_A+\vp_B+\vp_1$, then the identity (\ref{notion of weak solution when constructing the example}) takes the form 
	$$\tilde{\omega}=T(\tilde{\omega}).$$
	
	{\sc Step 2}. The continuity and compactness of the operator $T$ follow easily from the preceding embedding theorems in a standard way.
	Indeed, if $\ooo_k$ is a~bounded sequence in $H^1_0(B_1)$, when by embedding theorems it is compact in the Lebesgue space $L^4(B_1)$, so we can assume w.l.o.g. that $\|\ooo_k-\ooo_*\|_{L^4}\to 0$ for some $\ooo_*\in H^1_0(B_1)$, but then by above estimates 
$$\|\we_{B,k}-\we_{B_*}\|_{L^\infty(B_1)}\to0,$$ and this implies that linear functionals considered above converge strongly in the conjugate space to $H^1_0(B_1)$, etc. (see, e.g., \cite[Chapter~3]{leraybook} for the details). 
	
	\medskip
	{\sc Step 3}. Now it remains to prove only an~a\,priory estimate necessary for the application of the~Leray--Schauder fixed point theorem. Namely, we have to prove that 
	if $\tilde{\omega}_\lambda\in H^1_0(B_1)$ solves $\tilde{\omega}_\lambda=\lambda T(\tilde{\omega}_\lambda)$ for some real parameter~$\lambda\in[0,1]$, then 
	$\tilde{\omega}_\lambda$ is bounded in $H^1_0(B_1)$ uniformly with respect to $\lambda\in[0,1]$. 
	
	By definition, $\tilde{\omega}_\lambda$ satisfies for any $\eta\in H^1_0(B_1)$:
	$$\int_{B_1}\nabla\tilde{\omega}_\lambda\cdot\nabla\eta=\lambda\int_{B_1}\tilde{\omega}_\lambda(\alpha \mathbf{w}_A+\mathbf{w}_{B,\lambda})\cdot\nabla\eta+\lambda\int_{B_1}\omega_1 (\mathbf{w}_{B,\lambda}\cdot\nabla\eta)-\lambda\int_{B_1}\nabla\omega_1\cdot\nabla\eta.$$
	Taking $\eta=\tilde{\omega}_\lambda$ in the last identity, the first (and the largest!) term in the right hand side disappear, and we have from (\ref{linfty forn wb}) that  
	$$\begin{aligned}
		\|\nabla\tilde{\omega}_\lambda\|_{L^2(B_1)}^2\le &\lambda \|\mathbf{w}_{B,\lambda}\|_{L^\infty(B_1)} \|\omega_1\|_{L^2(B_1)}\|\nabla\tilde{\omega}_\lambda\|_{L^2(B_1)}+\lambda\|\nabla\omega_1\|_{L^2(B_1)}\|\nabla\tilde{\omega}_\lambda\|_{L^2(B_1)}\\[3pt]
		\le &C\lambda\left(\|\tilde{\omega}_\lambda\|_{L^4(B_1)}+\|\omega_1\|_{L^4(B_1)}\right) \|\omega_1\|_{L^2(B_1)}\|\nabla\tilde{\omega}_\lambda\|_{L^2(B_1)}+\lambda\|\nabla\omega_1\|_{L^2(B_1)}\|\nabla\tilde{\omega}_\lambda\|_{L^2(B_1)}\\[3pt]
		\le &C\lambda\biggl[\|\nabla\ooo_\lambda\|^2_{L^2(B_1)} \,\|\om_1\|_{L^2(B_1)} +\|\nabla\ooo_\lambda\|_{L^2(B_1)} \biggl(\|\om_1\|_{L^4(B_1)}\,\|\om_1\|_{L^2(B_1)}+\|\nabla\om_1\|_{L^2(B_1)}\biggr)\biggr].
	\end{aligned}$$
	Hence when $C\|\omega_1\|_{L^2(B_1)}\le \frac{1}{2}$, we have that 
	\begin{equation}\label{estimate for nabla tilde omega}
		\|\nabla\tilde{\omega}_\lambda\|_{L^2(B_1)}\le C\left(\|\omega_1\|_{L^4(B_1)}\|\omega_1\|_{L^2(B_1)}+\|\nabla\omega_1\|_{L^2(B_1)}\right).
	\end{equation}
	Hence by the Leray-Schauder fixed point theorem, there exists at least one solution to (\ref{vor-ns-form}). Moreover, by setting $\lambda=1$, we have 
	$$\|\nabla\omega\|_{L^2(B_1)}\le  C\left(\|\omega_1\|_{L^4(B_1)}\|\omega_1\|_{L^2(B_1)}+\|\nabla\omega_1\|_{L^2(B_1)}\right).$$
\end{proofof5}

\section{A counterexample demonstrating the estimates sharpness}
\label{sec:5-fin}

In this section, we construct the solution to (\ref{vor-ns-form}) such that
$$\left|\int_{B_1}\frac{\omega z^\bot}{|z|^2}\right|>>\left(\int_{B_1}|\omega|^2\right)^{\frac{1}{2}}.$$
From now on until the end of the section we fix $\psi_A=xy$ and $\www_A=\nabla\psi_A=(-x,y)$ with $|\mathbf{w_A}|=\sqrt{x^2+y^2}=r$. 
Since the proof is based on Theorem~\ref{ex-omega1}, 
we start from the construction of suitable boundary data $\omega_1$ that are constant along the level lines of $\psi_A$, thereby satisfying assumptions~(\ref{apr-omeg1})$_2$.

To simplify various technical aspects, we restrict ourselves to the first quadrant~$\RRR:=\{(x,y)\in\R^2:x>0,y>0\}$, where both coordinates are positive. Furthermore, the support of our function~$\omega_1$  will lie within a set 
$$B_1^{t_0}=\{(x,y)\in B_1\cap\RRR: x\cdot y>t_0\}$$
for some suitable positive $t_0>0$; that is, we stay away from the origin, where the gradient and the level lines of $\psi_A$ degenerate. First of all, we need some simple observations concerning the functions constant on level sets of $\psi_A$. 

Let $f:\R\to\R\in L^\infty(\R)$ be a real-valued function of one variable satisfying 
\begin{equation}\label{f-ass}
\begin{cases}
			f(t)=0,\quad &t\in (-\infty, t_0]\cup\bigl[\frac14,+\infty\bigr),\\[4pt]
			f(t)\ge 0,& t\in\bigl[t_0,\frac14\bigr].\\
		\end{cases}
\end{equation}
For such a function $f$, we introduce a~corresponding function of two variables $\sigma_f:B_1\subset\R^2\to\R$ putting
\begin{equation}\label{ss-ff}
\sigma_f(x,y):=\begin{cases}
			f(x\cdot y),\quad & \mbox{if \ }(x,y)\in B_1\cap\RRR, \\[4pt]
			0,& \mbox{otherwise}.\\
		\end{cases}
\end{equation}
In particular, by construction $\supp\sigma_f\subset \overline{\BT}$. As usual, we write $F\sim G$ iff there are some universal positive constants $c_1$, $c_2$ such that $c_1 F\le G\le c_2 F$.  

\begin{lemma}\label{prel-coarea}{\sl For any function $f\in L^\infty(\R)$ satisfying~(\ref{f-ass}) the following estimates hold:
\begin{equation}\label{est1-c}
	\int_{B_1}\sigma^2_f\ \sim\ \int\limits_{t_0}^{\frac{1}{4}}f^2(t)\ln\frac{1}{t}\,dt,
\end{equation}
\begin{equation}\label{est2-c}
	\left|\int_{B_1}\frac{\sigma_f \,z^\bot}{|z|^2}\right|\ \sim\  \int\limits_{t_0}^{\frac{1}{4}}\frac{f(t)}{\sqrt{t}}\,dt.
\end{equation}}\end{lemma}
Note, that the corresponding constants in the equivalence relations here are independent of $f$, $t_0$, etc.

\begin{proof}
Recall, that by Coarea formula, for any function $g\in L^\infty(B_1)$ the  equality
$$\int\limits_{\BT}g(x,y)\,dxdy=\int\limits_{t_0}^1\biggl(\int\limits_{(x,y)\in\BT, \ xy=t}\frac{g(x,y)}{r}\,ds\biggr)\,dt$$
holds, where $r=\sqrt{x^2+y^2}$ and $ds$ means the usual integration with respect to length (=$1$-Hausdorff measure). Then the estimates (\ref{est1-c})--(\ref{est2-c}) follow easily from the elementary properties of the regular level sets (=curves) \,$\{xy=t:t> t_0\}$ within $\BT$. 
\end{proof}

\begin{lemma}\label{lemma 11} {\sl For any constant $M_0>0$ there exists a~parameter $t_0\in \bigl(0,\frac14\bigr)$ and a~smooth function $f\in C^1(\R)\cap L^\infty(\R)$ satisfying~(\ref{f-ass})  such that $f\not\equiv0$ and 
\begin{equation}\label{est3-c}M_0\int_{B_1}\sigma^2_f< \left|\int_{B_1}\frac{\sigma_f \,z^\bot}{|z|^2}\right|^2.
\end{equation}}
\end{lemma}
\begin{proof}
Note that by~(\ref{est1-c})--(\ref{est2-c}) and H\"older inequality, 
$$\left|\int_{B_1}\frac{\sigma_f \,z^\bot}{|z|^2}\right|^2\sim \biggl(\int\limits_{t_0}^{\frac{1}{4}}\frac{f(t)}{\sqrt{t}}\,dt\biggr)^2\le 
\left(\int_{t_0}^{\frac{1}{4}}\frac{1}{t\ln\frac{1}{t}}\,dt\right)\cdot\left(\int_{t_0}^{\frac{1}{4}}|f(t)|^2\ln\frac{1}{t}\,dt\right)\sim
\left(\int_{t_0}^{\frac{1}{4}}\frac{1}{t\ln\frac{1}{t}}\,dt\right)\cdot \int_{B_1}\sigma^2_f.$$
Noting that $$M_0:=\int_{t_0}^{\frac{1}{4}}\frac{1}{t\ln\frac{1}{t}}\,dt\sim\ln\ln\frac{1}{t_0}\to\infty\quad\text{as }\ t_0\to0.$$
So, if we take 
\begin{equation}\label{est5-c}f(t)=\begin{cases}
			\frac1{\sqrt{t}\ln\frac{1}{t}},& t\in\bigl[t_0,\frac14\bigr],\\[4pt]
			0,\quad &t\in (-\infty, t_0]\cup\bigl[\frac14,+\infty\bigr)\\
		\end{cases}
\end{equation}
then the H\"older inequality turns out to be equality, and we have the required estimate 
$$M_0\int_{B_1}\sigma^2_f\sim \left|\int_{B_1}\frac{\sigma_f \,z^\bot}{|z|^2}\right|^2, $$
where $M_0\to +\infty$ as $t_0\to 0$. 
Now it remains to make the function $f$ smooth and compactly supported in $\bigl[t_0,\frac14\bigr]$ using the standard one-dimensional mollifying procedure, thus we get the required result. 
\end{proof}

\medskip
Now the claim of desired Theorem~\ref{f-conj} follows directly from the following

\begin{lemma}\label{construction of omega1} {\sl For any $M_0>0$ there exist $t_0\in\bigl(0,\frac14\bigr)$, $\alpha>0$,  and a function $\omega_1\in H^1(B_1)$ having support in $\BT$ and satisfying~(\ref{apr-omeg1}) such that the corresponding solution~$\omega$ to~(\ref{vor-ns-form}) satisfies
\begin{equation}\label{final-est-eq}
\left|\int_{B_1}\frac{\omega z^\bot}{|z|^2}\right|^2> M_0\,\|\omega\|^2_{L^2(B_1)}.
\end{equation}
}
\end{lemma}
\begin{proof}	
Fix a big parameter $M_0>0$. Take a smooth non-negative compactly supported $f:\R\to\R$ with corresponding $t_0>0$ from the previous Lemma. 
Put $$\omega_1:= \sigma_f.$$
By construction, $\omega_1\in H^1(B_1)$,  \ $\www_A\cdot\nabla\omega_1\equiv0$, \ and \ $\supp\omega_1\subset\overline{\BT}$; moreover, from~(\ref{est3-c}) we get
\begin{equation}\label{fee-1}
\left|\int_{B_1}\frac{\omega_1 z^\bot}{|z|^2}\right|^2> M_0\,\|\omega_1\|^2_{L^2(B_1)}.
\end{equation}
Multiplying $\omega_1$ by a small positive constant, we keep all the previous properties having in addition that $\|\omega_1\|_{L^2}$ is small, i.e., the assumptions~(\ref{apr-omeg1}) are fulfilled as well. So we can apply Theorem~\ref{ex-omega1} with an increasing sequence of parameters $\alpha=\alpha_k\to+\infty$. 
Consider the corresponding solutions $\omega_k$ to the vorticity NS-system~(\ref{vor-ns-form}). By assertion of Theorem~\ref{ex-omega1} we have an~apriory estimate
\begin{equation}\label{fee-2}
\|\omega_k\|_{H^1(B_1)}\le C,
\end{equation}
so we can assume without loss of generality that 
\begin{equation}\label{fee-3}
\omega_{k}\rightharpoonup\omega_\infty\mbox{ \  in  \ }H^1(B_1)\mbox{ \  as  \ }k\to\infty 
\end{equation}
for some~$\omega_\infty\in H^1(B_1)$. 
Consequently, by Sobolev Imbedding Theorems, 
\begin{equation}\label{fee-4}
\|\omega_{k}-\omega_\infty\|_{L^q(B_1)}\to0\qquad\forall q\in[1,+\infty).
\end{equation}
By construction, 
$$\int_{B_1}\nabla\omega_{k}\cdot\nabla\eta-\alpha_k\int_{B_1}\omega_{k} \www_A\cdot\nabla\eta-\int_{B_1}\omega_{k} \mathbf{w}_{B,k}\cdot\nabla\eta=0,\quad\forall \eta\in H^1_0(B_1).$$
Dividing both sides by $\alpha_k$ and then letting $k\to\infty$, we obtain
$$\int_{B_1}\omega_\infty \www_A\cdot\nabla\eta=0,\quad \forall\eta\in H^1_0(B_1).$$
This means that $\omega_\infty  $ is constant along the stream lines $\psi_A=xy=t$. Moreover, by the compact embedding $H^1(B_1)\hookrightarrow L^2(\partial B_1)$, we know that $\omega_\infty=\omega_1$ on $\partial B_1$. Hence we get $\omega_\infty=\omega_1$;
in particular 
\begin{equation}\label{fee-5}
\|\omega_{k}-\omega_1\|_{L^q(B_1)}\to0\qquad\forall q\in[1,+\infty).
\end{equation}
The last convergence easily implies 
\begin{equation}\label{fee-6}
\|\omega_{k}\|_{L^2(B_1)}\to\|\omega_1\|_{L^2(B_1)},
\end{equation}
\begin{equation}\label{fee-7}
\left|\int_{B_1}\frac{\omega_k z^\bot}{|z|^2}\right|\to \left|\int_{B_1}\frac{\omega_1 z^\bot}{|z|^2}\right|.
\end{equation}
Thus, by (\ref{fee-1}) for the considered solutions $\omega_k$ to the vorticity NS-system~(\ref{vor-ns-form}) we have the required property 
\begin{equation}\label{fee-ii}
\left|\int_{B_1}\frac{\omega_k z^\bot}{|z|^2}\right|^2> M_0\,\|\omega_k\|^2_{L^2(B_1)}
\end{equation}
for $k$ large enough. That means, that both Lemma~\ref{final-est-eq} and Theorem~\ref{f-conj} are completely proved. 
\end{proof}

\begin{rem}\label{concl-rem}It is well known that a constant shift in vorticity does not alter the velocity mean values; in particular, the key identity~(\ref{representation of the difference of the average by omega}) remains valid if on the right-hand side we replace $\omega$ with $\omega-c$, where $c$ is an arbitrary constant. Consequently, the assertions of Theorems~\ref{first basic estimate in annulus theorem}--\ref{first basic estimate in balls theorem} also remain valid if we replace $D_\omega$ with 
$$D_{\omega,c}=\int_{\Omega_{r_1,r_2}}|\omega-c|^2$$ (for an arbitrary $c$) in their statements. The proofs proceed analogously, with obvious modifications.
By choosing an optimal $c$, one can use the classical variance instead of $D_\omega$ in the statements of Theorems~\ref{first basic estimate in annulus theorem}--\ref{first basic estimate in balls theorem}:
$$\tilde{D}_{\omega}=\int_{\Omega_{r_1,r_2}}|\omega-\bar\omega|^2,$$ 
where $\bar\omega$ is the corresponding vorticity mean value:
$$\bar\omega:=\frac1{\bigl|\Omega_{r_1,r_2}\bigr|}\ \int_{\Omega_{r_1,r_2}}\omega.$$
\end{rem}
\bigskip
\noindent
{\color{black}
 {\bf Acknowledgment}.
The authors are deeply grateful to Dr. Xiao Ren for the valuable discussions. In particular, he suggested first considering functions that are constant along the level sets of some fixed harmonic function, which turns out to be very useful in constructing counterexamples.}


\smallskip
\par\noindent
{\bf Conflict of Interest}. The authors declare that they have no conflict of interest.
\bibliographystyle{plain}

\end{document}